\documentclass[11pt, reqno]{amsart}

\usepackage{amsmath,amssymb,amsthm, epsfig}
\usepackage{mathalfa}
\usepackage{amsfonts}
\usepackage{amsmath}
\usepackage{amssymb}
\usepackage[mathscr]{euscript}

\usepackage{color}

\def \R {\mathbb{R}}
\def \div {\mathrm{div}}

\def \suchthat {\ \big | \ }
\def \J {\mathscr{J}}
\def \dist {\mathrm{dist}}

\def\RR{{\mathbb R}}

\def\Om{{\Omega}}

\def\dfrac{\displaystyle\frac}
 
\def\epsilon{\varepsilon}

\newtheorem{theorem}{Theorem}
\newtheorem{lemma}[theorem]{Lemma}
\newtheorem{proposition}[theorem]{Proposition}
\newtheorem{corollary}[theorem]{Corollary}

\theoremstyle{definition}

\theoremstyle{remark}
\newtheorem{remark}[theorem]{Remark}
\numberwithin{equation}{section}

\newcommand{\intav}[1]{\mathchoice {\mathop{\vrule width 6pt height 3 pt depth  -2.5pt
\kern -8pt \intop}\nolimits_{\kern -6pt#1}} {\mathop{\vrule width
5pt height 3  pt depth -2.6pt \kern -6pt \intop}\nolimits_{#1}}
{\mathop{\vrule width 5pt height 3 pt depth -2.6pt \kern -6pt
\intop}\nolimits_{#1}} {\mathop{\vrule width 5pt height 3 pt depth
-2.6pt \kern -6pt \intop}\nolimits_{#1}}}

\title[Singular Free Boundary Problems]{Free boundary problems involving singular weights}

\author[J. Lamboley]{Jimmy Lamboley}
\address{Sorbonne  Universit\'e,  Universit\'e
  Paris  Diderot,  CNRS,  Institut  de  Math\'ematiques  de Jussieu-Paris Rive
 Gauche, IMJ-PRG, F-75005, Paris, France}{}
\email{jimmy.lamboley@imj-prg.fr}

\author[Y. Sire]{Yannick Sire}
\address{Johns Hopkins University, Krieger Hall, N. Charles St., Baltimore, MD 21218}{}
\email{sire@math.jhu.edu}

\author[E.V. Teixeira]{Eduardo V. Teixeira}
\address{University of Central Florida, 4393 Andromeda Loop N, Orlando, FL 32816}{}
\email{eduardo.teixeira@ucf.edu}

\begin{document}



\begin{abstract}
In this paper we initiate  the investigation of free boundary minimization problems ruled by general singular operators with $A_2$ weights. We show existence, boundedness and continuity of minimizers. The key novelty is a sharp $C^{1+\gamma}$ regularity result for  solutions at their singular free boundary points. We also show a corresponding non-degeneracy estimate. 

\end{abstract}
\maketitle

\tableofcontents

\noindent{\bf Aknowledgement:} The authors would like to thank the hospitality of the Univ. Federal do Ceara in Fortaleza, where this work was initiated, and the Brazilian-French Network in Mathematics as well. This work was also supported by the projects ANR-18-CE40-0013 SHAPO financed by the French Agence Nationale de la Recherche (ANR) and the Brazilian National Council for Scientific and Technological Development,CNPq, under MCTI/CNPQ/Universal 14/2014 - Faixa C project No 441632/2014-9. Y.S. would like to thank the Simons foundation. 

\section{Introduction}

We study local minimizers of singular, discontinuous functionals of the form 
\begin{equation} \label{SingFB}
	\mathscr{J}(u, \Omega)= \int_\Omega  \left (\omega (x) |\nabla u |^2+ \chi_{\{u>0\}} \right ) dx \longrightarrow \text{min},
\end{equation}
where $\Omega$ is a bounded domain of $\R^d$, $d\ge 2$, and $\omega(x)$  is a measurable, singular $A_2$ weight in the sense that $0\le  \omega \le +\infty$, both $\omega$ and $\omega^{-1}$ are locally integrable, and
$$	
	 \left ( \int_{B_r(z_0)} \omega \right) \left (  \int_{B_r(z_0)} \omega^{-1} \right ) \le Cr^{2d} ,
$$
for all balls $ B_r(z_0) \subset \Omega$. Thus, $\omega$ may become zero or infinity along a lower-dimensional subset of $\Omega$, hereafter denoted by: 
$$
	\Lambda_0 (\omega) := \omega^{-1}(0), \quad \Lambda_\infty(\omega) := \omega^{-1}(+\infty).
$$
It will also be convenient to denote $\Lambda(\omega) := \Lambda_0 (\omega)  \cup \Lambda_\infty (\omega)$. The class of $A_2$ weight functions was introduced by  Muckenhoupt \cite{muck} and is of central importance in modern harmonic analysis and its applications. A canonical example of an $A_2$ function is $|x|^\alpha$  with $-d < \alpha < d$ --- having an isolated singularity at the origin. Recent results, see for instance \cite{cafS},  show  that $A_2$ weights play an important role in the theory of non-local diffusive problems. In particular, when $\omega(x) = |x_1|^\beta$, $-1 < \beta < 1$, and the axis $\left \{ x_1=0 \right \}$ is a subset of $\partial \Omega$, then \eqref{SingFB} falls into the case studied in \cite{CRS}, where fractional cavitation problems are considered. 
 
The mathematical analysis of free boundary cavitation problems goes back to the pioneering work of Alt and Caffarelli \cite{AC}, corresponding to the case $\omega \equiv 1$ in \eqref{SingFB}.  In the present work, we start the investigation of free boundary problems of the type \eqref{SingFB}, for possibly singular weights, that is when $\Lambda_\infty(\omega) \not = \emptyset$. Similarly, we can treat the degenerate case, i.e. when $\Lambda_0(\omega) \not = \emptyset$. However, for didactical purposes, in this paper we shall restrict the analysis to singular weights.

We show existence of a local minimizer $u$, and analyze analytic and weak geometric properties of the free boundary $\partial \left \{ u>0 \right \} \cap \Omega$. The latter task is, in principle, a delicate issue. For instance, one notices, for singular weights, the existence of two distinct types of free boundary points: 
\begin{enumerate}
	\item $\partial \left \{ u>0 \right \} \cap \Lambda_\infty  = \emptyset $;
	\item $\partial \left \{ u>0 \right \} \cap \Lambda_\infty  \neq \emptyset $.
\end{enumerate}

Case $(1)$ refers to a non-homogeneous version of the Alt-Caffarelli problem, \cite{AC}, as, away from the singular set, the weight is uniformly elliptic, see for instance \cite{deSilva, dosPT, FS, T1}. Case (2) is rather more delicate, and in particular, a central result we prove in this current work classifies the geometric behavior of a local minimizer near a free boundary point $z_0 \in \partial \{u> 0\} \cap \Lambda_\infty$,  in terms of its singularity rate near $z_0$ --- a purely analytic information of the problem. Indeed, we show local minimizers are {\it precisely} $C^{1+ \gamma}$ smooth along their corresponding free boundaries, where $\gamma$ is half of the geometric blow-up rate of $\omega$ as it approaches the singular set $\Lambda_\infty$; see condition (H2) for precise definitions.

The paper is organized as follows. In Section \ref{sct math setup} we formally present the minimization problem we shall study. A brief description of the initial mathematical tools required in the investigation of $A_2$-singular free boundary problems is also delivered in that section. In Section \ref{sct Exist and Linfty bounds} we discuss existence and $L^\infty$ bounds for minimizers, whereas in Section \ref{sct comp} we establish various compactness properties for family of minimizers. Section \ref{section-homog} is devoted to some considerations related to homogenization. In Section \ref{sct sharp reg} we prove the key novel result of the paper, namely that solutions to $A_2$ cavitation problems are $C^{1+\gamma}$ regular at their singular free boundary points, where $\gamma$ is a sharp prescribed value. In particular, if $z_0$ is a free boundary point and $\omega(z_0)  < + \infty$, that is it is non-singular, then $\gamma(z_0) = 0$ and we recover the classical Alt-Caffarelli Lipschitz regularity estimate for cavitation problems.  Finally in Section \ref{sct nondeg} we obtain a quantitative non-degeneracy estimate for solutions near their singular free boundary points.

\section{Mathematical set-up}  \label{sct math setup}

In this section we give a precise description of the minimization problem considered in this article and gather some of the main known results about elliptic equations involving $A_2$ weights, required in our study.

Given an open set $\Omega \subset \mathbb{R}^d$, we denote by $\mathcal{M}(\Omega)$ the set of all real-valued measurable functions defined on $\Omega$. A nonnegative locally integrable function $\omega \colon \Omega \to \mathbb{R}$ is said to be an $A_2$ weight if $\omega^{-1}$ is also locally integrable and  
\begin{equation}\label{A2 cond}
	\sup_{B \subset \Omega } \Big ( \frac{1}{|B|}\int_B \omega \Big ) \Big (\frac{1}{|B|}\int_B \omega^{-1} \Big ) \leq C_1,
\end{equation}
holds for a constant $C_1>0$ and any ball $B\subset \Omega$. Two weights are said to belong to the same $A_2$ class if condition \eqref{A2 cond} is verified for both functions with the same constant $C_1$.

For an $A_2$ weight $\omega$ and $1\le p < \infty$, we define
$$
	L^p(\Omega, \omega) = \left \{ f \in \mathcal{M}(\Omega) \big | ~ \|f\|_{L^p(\Omega, \omega)} := \left ( \int_{\Omega} |f(x)|^p \omega(x) dx \right )^{1/p} < \infty \right \}.
$$ 
Accordingly, we define the weighted Sobolev space as
$$
	W^{1,p}(\Omega, \omega) := \left \{ u \in L^p(\Omega, \omega) \big | ~ D_i u \in L^p(\Omega, \omega), \text{ for } i=1,2,\cdots d \right \},
$$
and, by convention, we write $W^{1,2}(\Omega, \omega)$ as $H^1(\Omega, \omega)$.

An $A_2$ weight $\omega$ gives raise to the degenerate/singular elliptic operator 
$$
	L_\omega (\cdot)  = \div(\omega \nabla \cdot),
$$ 
which acts on $H^1(\Omega, \omega)$.  In a series of three papers, \cite{FKS,FJK, FJK2}, 
Fabes, Jerison, Kenig, and Serapioni developed a systematic theory for this class
of operators: existence of weak solutions, Sobolev embeddings, Poincar\'e inequality, Harnack
inequality, local solvability in H\"older spaces, and estimates on the
Green's function. Below we state four results from \cite{FKS}, see also \cite{HKM} supporting, directly or indirectly, the framework developed in this present article.

\begin{theorem} [Solvability in Sobolev spaces] \label{solveFKS}
Let $\Omega\subset\RR^{d}$ be a smooth bounded domain, 
$h=(h_1,...,h_{n})$ satisfying $|h|/\omega \in L^2(\Omega,\omega)$, and 
$g \in H^1(\Omega,\omega)$. Then, there exists a unique solution  
$u\in H^1(\Omega,\omega)$ of 
$L_\omega u=-{\rm div }\, h$ in $\Omega$ with $u-g \in H^1_0(\Omega,\omega)$.
\end{theorem}

\begin{theorem} [Local H\"older regularity]\label{HolderFKS}
Let $h$ be a weak solution of $L_\omega h= 0$ in 
$B_1 \subset \mathbb{R}^d$.  Then, $u$ is H\"older continuous in $B_{1/2}$ with a H\"older exponent $\mu$ depending only on $d$ and the $A_2$ class of $\omega$. Furthermore
$$
	\|h\|_{C^\mu(B_{1/2})} \le M \left ( \frac{1}{\omega(B_1)} \int h^2(x) \omega(x) dx \right )^{1/2},
$$
for a constant $M$ also depending only on $d$ and the $A_2$ class of $\omega$.
\end{theorem}

\begin{theorem} [Harnack inequality] \label{HarnackFKS}
Let $u$ be a positive solution of $L_\omega u=0$ in 
$B_{4R}(x_0)\subset\RR^{d}$. Then, $\sup_{B_R(x_0)} u \leq C \inf_{B_R(x_0)} u$
for some constant $C$ depending only on $d$ and the $A_2$ class of $\omega$ --- and in particular, independent of~$R$.
\end{theorem}

\begin{theorem}[Poincar\'e inequality] \label{PoincFKS} There is as positive constant $C>1$ such that for all function $f$ in $H^1(B_R, \omega)$ satisfying $f = 0$ on $\partial B_R$, the following holds
$$
	 \int_{B_R}|f|^2 \omega dx   \leq C R^2  \int_{B_R}|\nabla f|^2 \omega dx.
$$ 
\end{theorem}

%%%%%%%%%%%%%%%%%%%%%%%%%%%%%%%%%%%%%%%%%%%%%%%%%%%%%%%%%%%%%%%

Let us now turn to the mathematical description of the problem to be studied in the present article. Given a nonnegative boundary datum $f \in H^1(\Omega, \omega) \cap L^\infty(\Omega)$, we consider the minimization problem
\begin{equation}\label{degFB}
	\J(\omega, u, \Omega)= \int_\Omega  \left (\omega (x) |\nabla u |^2+\chi_{\{u>0\}} \right ) dx  \longrightarrow \text{min},
\end{equation}
among functions $u \in H^1_f(\Omega, \omega):=f+H^1_0(\Omega,\omega)$, where $\chi_{O}$ stands for the characteristic function of the set $O$ and $\omega$ is an $A_2$ weight. 

We are mostly interested in local geometric properties of local minima near a singular free boundary point. Henceforth, as to properly carry out the analysis, it is convenient to localize the problem into the unit ball $B_1$ and assume the origin is a free boundary point, i.e., $0 \in \partial \{u>0\}.$ In addition we shall assume throughout the paper the following structural condition on the weight $\omega$:
\begin{itemize}

	\item[(H1)] The weight $\omega$ belongs to  $A_2$-class and $0 < \tau_0 \le \omega \le +\infty$  a.e. in $\Omega$.
\end{itemize}

It turns out that (H1) prescribes minimal condition under which one can develop an existence and regularity theory for corresponding singular cavitation problem. 

 We comment that we have chosen to develop the analysis of problems involving singular weights, rather than degenerate ones. In turn, we consider strictly positive weights that may blow-up along its singular set $\Lambda_\infty(\omega)$. We could similarly treat bounded, degenerate weights of order $\gtrsim r^{\sigma}$, for some $\sigma < 2$.  
 %The core, for our purposes, is to prevent singular and degenerate points from becoming arbitrarily close, that is, a minimal lower bound on the distance between $\Lambda_\infty$ and $\Lambda_0$ has to be prescribed.

A number of physical free boundary problems fall into the above mathematical set-up. Typical examples of weight functions we have in mind are 
$$
	\omega(x) = |x'|^\alpha, 
$$
where $x=(x', x_{m+1}, \cdots, x_d)$, for $0\le m< d$ and $-m< \alpha \le 0$.  More generally, if $\mathcal{N}$ is an $m$-dimensional manifold properly embedded in $\mathbb{R}^d$, $0\le m< d$,  we are interested in weights of the form
$$
	\omega(x) = \text{dist}(x, \mathcal{N})^\alpha,
$$
for some $-m< \alpha \le 0$. This class of weight functions gives raise to the analysis of free boundary problems ruled by diffusion operators with an $m$-dimensional singular set. Anisotropic weights of the form
$$
	\omega(x) := \prod_{i=1}^d | x_i|^{\alpha_i}, 
$$
also fall under the hypothesis considered in this work. In this case, condition (H1) is verified as long as $-1< \alpha_i\le 0$.

We are further interested in weights with possible distinct behaviors along sets of different dimensions, say 
$$
	\omega(x) = |(x_1, x_2, \cdots, x_m)|^{\alpha_1} \cdot |(x_{m+1}, \cdots, x_d)|^{\alpha_2},
$$ 
where $-m<\alpha_1 \le 0$, $-d+m< \alpha_2 \le 0$. In this model,  $0$ is a singular free boundary point of degree $|\alpha_1 + \alpha_2|$, as we shall term later. If we label the cones $\mathscr{C}_1 := \left \{ \Pi_{i=1}^m x_i = 0  \right \} $ and $\mathscr{C}_2 := \left \{ \Pi_{i=m+1}^d x_i = 0  \right \} $, then any free boundary point in $\mathscr{C}_1 \setminus \mathscr{C}_2$ is singular of degree $|\alpha_1|$ and similarly, a point in $\partial \{ u > 0 \} \cap \left (\mathscr{C}_2 \setminus \mathscr{C}_1 \right )$ is  singular of degree $|\alpha_2|$.  Any other free boundary point $z \in \partial \{u> 0 \} \setminus \left ( \mathscr{C}_1  \cup \mathscr{C}_2 \right )$ is a free boundary point of degree $0$.  

Notice that no continuity assumption has been required on $\omega$. In particular, we are interested in the analysis of free boundary problems in possibly random media. Given a measurable function $0 < c_0 \le \theta < c_0^{-1}$ defined on the unit sphere $\mathbb{S}^{d-1}$, $0\le m< d-2$ and $-m< \alpha \le 0$, we can always define an $\alpha$-homogeneous, $A_2$ weight function $\omega$ in $B_1 \setminus \{0\}$ as
$$
	\omega(x) := |x^\prime |^\alpha \cdot  \theta \left ( \frac{x}{|x|} \right ),
$$
where, as above, $x=(x', x_{m+1}, \cdots, x_d)$.
This is another important example of weights we have in mind as to motivate this work.

 We conclude this section by setting a nomenclature convention: hereafter any constant that depends only upon dimension, $d$, and $\omega$ will be called universal.  

\section{Existence and local boundedness} \label{sct Exist and Linfty bounds}

In this preliminary section, we discuss existence and local boundedness of minimizers. We also comment on the Euler-Lagrange equation satisfied by a local minimum. The proofs follow somewhat classical arguments, thus we simply sketch them here for sake of completeness. 

Given a nonnegative boundary datum $f \in H^1(\Omega,\omega)$, one can consider a minimizing sequence $v_j \in H_f^1(\Omega,\omega)$. Clearly,
$$
	C > \J (v_j) \ge \int_{\Omega} \omega (x) |\nabla v_j|^2 dx.
$$
Thus, up to a subsequence, $v_j \to u$ weakly in $H^1(\Omega,\omega)$, for some function $u\in H_f^1(\Omega,\omega)$. By compactness embedding, $v_j \to u$ in $L^2(\Omega, \omega)$ and  $v_j \to u$ a.e. in $\Omega$. Passing to the limit as $j \to \infty$ (see \cite[Section 1.3]{AC} to handle the term $\chi_{\{v_{j}>0\}}$), one concludes 
$$
	\J(u) \le \liminf\limits_{j\to \infty} \J (v_j) = \min \J.
$$
This shows the existence of a minimizer. To verify that $u$ is nonnegative, one simply compares $u$ with $u^{+}$ in the minimization problem.  Also, if the boundary datum $f$ is assumed to be in $H^1(\Omega,\omega) \cap L^\infty(\Omega)$, then for each $ |t| \ll 1$, the function $u + t \left (\|f\|_\infty - u\right )^{-}$ competes with $u$ in the minimization problem. Standard computations yield $\{ u > \|f\|_\infty\}$ has measure zero, that is, $\{0 \le u \le \|f\|_\infty\}$ has total measure in $\Omega$.

Now, if $\varphi$ is a nonnegative test function in $C_0^\infty(\Omega)$, then $u+\varphi$ competes with $u$ in the minimization problem. As  $\{ u+ \varphi > 0 \} \supset \{u > 0 \}$, there holds,
$$
	-2\int_{\Omega} \omega(x) \nabla u \cdot \nabla \varphi dx = \J(u+\varphi) - \J(u) - \int_{\Omega} \left ( \chi_{\{ u+ \varphi > 0 \}}  - \chi_{\{u > 0 \}}  \right ) dx \ge  0.
$$
This shows $\div \left (\omega(x) \nabla u \right )$ defines a nonnegative measure $\nu$.  If $B_\delta(x_0) \subset \{u > 0 \}$, then given a test function $\varphi \in C_0^\infty(B_\delta(x_0))$, for $0< |t| \ll 1$, $u+t \varphi$ is also positive in  $B_\delta(x_0) $, thus we conclude $\nu \equiv 0$ in the interior of  $ \{u > 0 \}$.

\smallskip

Let us gather the information delivered above and state as a theorem for future reference.

\begin{theorem}\label{ext}
Let $\omega$ be any $A_2$ weight and $f \in H^1(\Omega,\omega)$ nonnegative. Then there exists a minimizer $u \in H^1(\Omega,\omega)$ to problem \eqref{degFB} such that $u=f$ on $\partial \Omega$, in the trace sense. Furthermore, $u$ is nonnegative, $\|u\|_{L^\infty(\Omega)}\le  \|f\|_{L^\infty(\Omega)}$, and there exists a non-negative Radon measure $\nu$, supported on its free boundary $\partial \{u>0\}$, such that
$$
	\text{div} \left (\omega(x) \nabla u \right ) = \nu
$$ 
is verified in the distributional sense. 
\end{theorem}

\section{Universal Continuity} \label{sct comp}

In this section we establish that local minima of $\J$ are locally H\"older continuous, with universal bounds. In particular, any family of bounded local minima is pre-compact in the uniform convergence topology. 
Hereafter in this section we shall use the notation
\begin{equation}\label{notation avr}
	\intav{B_r(y))} f(x) \omega(x) dx := \frac{1}{\omega\left (B_r(y) \right )} \int_{B_r(y)} f(x) \omega(x) dx.
\end{equation}

We also write, for a measurable set $A$, $\omega(A) := \int_A \omega(x) dx$. Recall $A_2$-weights are doubling measures, i.e. there exists $D>0$ such that
\begin{equation}\label{doubling}
	0< \omega(B_{2r}(y)) \le D  \omega(B_{r}(y)),
\end{equation}
for all $y$ and $r>0$, see for instance \cite{HKM}.

We start off by investigating the scaling of the problem. Let $u$ be a local minimizer of 
$$
	\J(\omega, u, B_1) := \int_{B_1} \omega(x) |\nabla u (x)|^2+\chi_{\{u>0\}} dx.
$$
For $x_0 \in B_{1/2}$, $0< \rho < 1/2$ and $t>0$, define the re-scaled function $\tilde{u} \colon B_1 \to \mathbb{R}$ by
$$
	\tilde{u}(y) = t \cdot u(x_0 + \rho y).
$$
If we define the map $\Phi \colon B_1 \to B_{\rho}(x_0)$ as $z= \Phi(y) =  x_0 + \rho y$, for a subdomain $ \Om \subset \Phi(B_1)$, change of variables yields 
\begin{eqnarray*}
	\J(\omega, u, \Om)&=&\int_{\Phi^{-1}( \Omega)}\left [ \omega(x_0 + \rho y) |\nabla u(x_0 + \rho y)|^2+\chi_{ \{u(x_0 + \rho y)>0\}}\right] \rho^d dy\\
	&=&\int_{\Phi^{-1}(\Omega)} \left[\omega (x_0 + \rho y) t^{-2} \rho^{-2} |\nabla \tilde{u} |^2+\chi_{\{\tilde{u} >0\}} \right]\rho^d dz \\
	&=& \rho^{d-2} t^{-2} \int_{\Phi^{-1}(\Omega)} \left[{\omega} (x_0 + \rho y)  |\nabla \tilde{u} |^2+t^{2} \rho^{2} \chi_{\{\tilde{u} >0\}} \right] dz .
\end{eqnarray*}
Hence, we conclude that $\tilde{u}$ is a local minimizer of a functional $\tilde{\J}$ ruled by a weighted function $\tilde{\omega}$ in the same $A_2$ class of  $\omega$, but with jump forcing term $t^{2} \rho^{2} \chi_{\{\tilde{u} >0\}}$. 

Given a local bounded minimizer $u$ of $\J$ in $B_1$, for a fixed $0< \epsilon < 1$, select 
$$
	\rho = \sqrt{\epsilon \intav{B_1} u^2 \omega dx} \quad \text{and in the sequel} \quad t = \rho^{-1} \sqrt{\epsilon}.
$$
The corresponding scaled function $\tilde{u}$ then verifies $ \intav{B_1} \tilde{u}^2  \omega dx  \le 1$ and is a local minimizer of 
$$
	{\J}^\epsilon (v) := \int \tilde{\omega}(x) |\nabla v (x)|^2+ \epsilon \chi_{\{v>0\}} dx,
$$
where $\tilde{\omega}(x)$ lies in the same $A_2$ class as $\omega$. 

Notice that proving H\"older continuity for $u$ is equivalent to proving $\tilde{u}$ is  H\"older continuous; any estimate proven for $\tilde{u}$ yields a corresponding one for $u$, with proper constants adjustments, depending only on the choices for $t$ and $\rho$. In particular, should these choices be universal, the corresponding result for $u$ will be also universal. 

The above discussion motivates the following result, which measures how much a local minimizer of a functional with small jumping force deviates from its $\omega$-harmonic replacement. We shall need to work with a slightly more general class of operators.

\begin{proposition} \label{Closeness}
Let $\beta \colon \mathbb{R} \to [0,1]$ be a measurable function and $u$ be a local minimizer of 
$$
	\J^{\epsilon, \beta}(v) =  \int \omega(x) |\nabla v (x)|^2+ \epsilon  \beta({v}) dx,
$$
where $\omega$ is in the $A_2$-class. Let $h$ be the unique weak solution of  
$$
	\left \{
		\begin{array}{rlll}
			\text{div}\left ( \omega(x) \nabla h \right  ) &=& 0 &\text{ in } B_{R}(y)\\
			h & =& u & \text{ on }  \partial B_{R}(y),
		\end{array}
	\right.
$$
where $y\in B_{1/2}$ and $0<R<1/2$. Then 
$$
	\displaystyle  \int_{B_R(y)} \omega(x) \left |\nabla u(x) -  \nabla h(x) \right |^2 dx  \le  \epsilon R^d.
$$
\end{proposition}
\begin{proof}
Since $h$ competes with $u$ in the minimization problem, 
$$
	\J^{\epsilon, \beta}(u) \le \J^{\epsilon, \beta}(h),
$$
which yields 
$$
	\int_{B_{R}(y)} \omega(x) \left ( |\nabla u(x)|^2 - |\nabla h(x)|^2  \right ) dx \le \epsilon \int_{B_{R}(y)} \left (   \beta({h}) - \beta (u) \right ) dx  \le \epsilon R^d.
$$
Now, from the PDE satisfied by $h$, we have
\begin{equation}\label{thm reg0 eq3}
	\int_{B_{R}(y)} \omega(x)  \nabla h \cdot \nabla (h-u) dx = 0,
\end{equation}
hence,
\begin{equation}\label{thm reg0 eq4} 
	\int_{B_{R}(y)} \omega(x) \left |\nabla u - \nabla h \right |^2 dx =  \int_{B_{R}(y)} \omega(x) \left ( |\nabla u|^2 -  |\nabla h|^2 \right ) dx,
\end{equation}
and the Proposition is proven.
\end{proof}

 Finally from Theorem \ref{HolderFKS}, $\omega$-harmonic functions are locally H\"older continuous. More precisely, if $h$ is a weak solution of $\text{div} (\omega(x) \nabla h) = 0$ in $B_{1/2}$, and say 
$$
	\left ( \intav{B_{1/2}} h^2(x) \omega(x) dx \right )^{1/2} \le 1
$$ 
then for some $0<\mu < 1$ and $\Lambda$, depending only on dimension and $A_2$ class of $\omega$  there holds
\begin{equation} \label{Plain Holder}
	\|h\|_{C^{0,\mu}(B_{1/4})} \le \Lambda.
\end{equation}

In what follows we shall argue along the lines of \cite{T2.5}. 

\begin{lemma} \label{1st step} Assume (H1). There exist universal constants $0<\epsilon<1 $, $0<\lambda<1/4$ and $ |a_0| < \Lambda$ such that if $u$ is a local minimizer of $\J^{\epsilon, \beta}$, satisfying $\intav{B_1} |u|^2\omega(x) dx \le 1$, then
$$
	  \intav{B_\lambda} \left | u - a_0 \right |^2 \omega(x) dx \le \lambda^\mu.
$$
\end{lemma}
\begin{proof} Let $h$ be the $\omega$-harmonic replacement of $u$ in $B_{1/2}$. From Proposition \ref{Closeness} and Poincar\'e inequality, we can estimate
\begin{equation}
	\int_{B_{1/2}} \left | u - h \right |^2 \omega(x) dx \le  C\epsilon.
\end{equation}
Because $  \intav{B_1} |u|^2\omega(x) dx \le 1$, we can make a first choice on the smallness of $\epsilon$ to ensure 
$$
\intav{B_{1/2}} h^2(x) \omega(x) dx  \le  \sqrt{(D+1)},$$
 where $D>0$ is the doubling constant from \eqref{doubling}. Indeed,
$$
	\begin{array}{lll}
		\|h\|_{L^2(B_{1/2}, \omega)} &\le&    \|h-u\|_{L^2(B_{1/2}, \omega)} + \|u\|_{L^2(B_{1/2}, \omega)} \\
		&\le &  \sqrt{ C\epsilon} +\|u\|_{L^2(B_{1}, \omega)} \\
		&\le &  \sqrt{ C\epsilon} + \sqrt{\omega(B_1)} \\
		& \le &  \sqrt{ C\epsilon} + \sqrt{{D}\omega(B_{1/2})} \\
		& \le & \sqrt{{(D+1)}\omega(B_{1/2})},  
	\end{array}
$$ 
provided $\epsilon >0$ is small enough. We now set $a_0 = h(y)$ and, for a $0<\lambda < 1/4$ to be chosen, estimate
\begin{equation}
	\begin{array}{lll}
		\displaystyle   \intav{B_{\lambda}} \left | u - a_0 \right |^2 \omega(x) dx &\le& \displaystyle  2 \left (  \intav{B_{\lambda }} \left | h - h(y) \right |^2\omega(x) dx + \intav{B_{\lambda }} \left | u - h \right |^2 dx \right ) \\
		&\le& \displaystyle  2 \left (  \Lambda (D+1) \cdot \lambda^{2\mu} + \frac{C\epsilon{\omega(B_1})}{\omega(B_\lambda)} \right ).
	\end{array}
\end{equation}
Now we choose $0<\lambda  \ll 1$ so small that
$$
	\Lambda(D+1) \cdot \lambda^{2\mu} = \frac{1}{4} \lambda^\mu, 
$$
and in the sequel select $\epsilon>0$ such that
$$
	\frac{C\epsilon{\omega(B_1})}{\omega(B_\lambda)} =  \frac{1}{4} \lambda^\mu.
$$
Note these choices are universal. Lemma \ref{1st step} is proven.
\end{proof}

We are ready to prove universal H\"older continuity for local minimizers.

\begin{theorem}[Local $C^{0,\tau}$] \label{reg0} Let $u$ be local minimizer of $\J$ in $B_1$. There exist universal constants $0<\tau \ll 1$ and $C>1$  such that $u \in C^{0,\tau}(B_{1/2})$ and
$$
	\|u\|_{C^{0,\tau}(B_{1/2})} \le C\|u\|_{L^2(B_{1})}.
$$
\end{theorem}

\begin{proof} The idea is to iterate Lemma \ref{1st step}. Let $u$ be a minimizer of $\J^\epsilon$,  where $\epsilon> 0$ is the number from Lemma \ref{1st step}. Note that from the discussion at the beginning of this Section, proving H\"older continuity for a local minimizer of $\J^\epsilon$ yields corresponding estimates for minimizers of the original functional $\J$. 

Define $u_1 \colon B_1 \to \mathbb{R}$ by
$$
	u_1(x) := \frac{1}{\lambda^{\mu/2}}\left (  u( \lambda x) - a_0 \right ).
$$
From Lemma \ref{1st step},  $\intav{B_1} u_1^2 \omega(x)dx \le 1$ and arguing as at the beginning of this Section, we deduce $u_1$ is a local minimizer of
$$
	\tilde{J} (v) := \int \tilde{\omega} (x) \left | \nabla v \right |^2 + \epsilon \lambda^{2(1-\mu/2)}  \chi_{\{ v >  \lambda^{\mu/2} a_0 \}} dx,
$$
where $\tilde{\omega}$ is a weight in the same $A_2$ class as $\omega$ and satisfies (H1) with the same constant $\tau_0$. Because $0< \mu < 1$,  $\epsilon \lambda^{2(1-\mu/2)} < \epsilon$,  we can argue as in Lemma  \ref{1st step} to find another constant $|a_1| < \Lambda$, such that
$$
	\intav{B_\lambda} \left | u_1 - a_1 \right |^2 \omega(x) dx \le \lambda^\mu,
$$
which yields for $u$ the estimate
$$
	\intav{B_{\lambda^2} } \left | u - \left ( a_0 + \lambda^{\mu/2} a_1 \right ) \right |^2 \omega(x)dx \le \lambda^{2\mu}.	
$$
Arguing now by induction, we conclude 
\begin{equation}\label{ind1}
	\intav{B_{\lambda^n}} \left | u - a_{n} \right |^2 \omega(x) dx \le \lambda^{n \cdot \mu}.	
\end{equation}
where
$$
	a_{n} = a_0 + \lambda^{\mu/2} a_1 + \cdots + \lambda^{(n-1)\mu/2}  a_{n-1}.
$$
One easily verifies that the above is a Cauchy sequence, and therefore convergent. Let $a_\infty = \lim\limits_{n\to \infty} a_n$. We estimate
\begin{equation}\label{ind2}
	|a_\infty - a_n| \le \lim\limits_{m\to \infty} \Lambda \cdot \sum\limits_{i=n}^m \lambda^{{(\mu/2)} i} \le \frac{\Lambda}{1-\lambda^{\mu/2}}  \cdot \lambda^{n\mu/2} 
\end{equation}
We can therefore estimate, from \eqref{ind1} and \eqref{ind2}
\begin{equation} \label{final01}
 	\intav{B_{\lambda^n} } \left | u - a_{\infty} \right |^2 \omega(x) dx \le \left (1 +  \left [ \frac{\Lambda}{1-\lambda^{\mu/2}}  \right ]^2 \right ) \lambda^{n \cdot \mu}.
 \end{equation}
Now, fixing $y \in B_{1/2}$ and working with $v(x) = u(y+x)$, we reach \eqref{final01} for  $B_{\lambda^n}(y)$.  Owing to \eqref{final01}, it is easy to see that, if $u_{B_{\lambda^n(y)}}$ denotes
$\frac{1}{\omega({B_{\lambda^n(y)}})} \int_{{B_{\lambda^n(y)}}} u(x) \omega(x) dx$, then
$$
	\begin{array}{lll}
	 	\displaystyle \intav{B_{\lambda^n(y)} } \left | u - u_{B_{\lambda^n(y)}} \right |^2 \omega(x) dx &\le& 2 \Big ( \displaystyle \intav{B_{\lambda^n(y)} } \left | u - a_\infty  \right |^2 \omega(x) dx  \\
		&+&   \displaystyle \intav{B_{\lambda^n(y)} } \left | u_{B_{\lambda^n(y)}} - a_\infty  \right |^2 \omega(x) dx \Big ) 
	 \end{array}
 $$
 Furthermore,
 $$
 	\begin{array}{lll}
 		\displaystyle \intav{B_{\lambda^n(y)} } \left | u_{B_{\lambda^n(y)}} - a_\infty  \right |^2 \omega(x) dx  &=& \displaystyle \intav{B_{\lambda^n(y)} } \left | \intav{B_{\lambda^n(y)} } (u-a_\infty) \omega(x) dx \right |^2 \omega(x) dx \\
		&\le& \displaystyle \intav{B_{\lambda^n(y)} } \left | u-a_\infty \right |^2 \omega(x) dx,
	\end{array}
 $$
 by Jensen's inequality. Thus, we conclude that 
 \begin{equation} \label{final02}
 	\intav{B_{\lambda^n} } \left | u - u_{B_{\lambda^n(y)}} \right |^2 \omega(x) dx \le 4 \left (1 +  \left [ \frac{\Lambda}{1-\lambda^{\mu/2}}  \right ]^2 \right ) \lambda^{n \cdot \mu}.
 \end{equation}
Finally, as noted earlier, the measure $A \mapsto \int_A \omega(x) dx$ is doubling and clearly continuous in the sense that 
$$
	\lim\limits_{y \to x} \omega \left ( B_r(x) \triangle B_r(y) \right ) = 0,
$$
for all $r>0$. Thus, applying Campanato's theorem for doubling measures, see e.g. \cite[Theorem 3.2]{G}, we conclude  $u$ is H\"older continuous  and the proof of Theorem \ref{reg0} is complete.
 \end{proof}

%%%%%%%%%%%%%%%%%%%%

\section{Homogenization}\label{section-homog}
We now turn our attention to limiting free boundary problems arising from homogenization. That is, hereafter we assume
\begin{itemize}
\item[(H2)] Let $0$ be a free boundary point. There exists $-d < \alpha \le 0$ such that, as $\lambda \to 0+$,
	$$
		\omega_\lambda( x) := \lambda^{|\alpha|} \omega (\lambda x),
	$$
	converges locally in $L^1$ to a weight $\omega_0(x) \in A_2$.
\end{itemize}

As we shall demonstrate, condition (H2) yields an elegant solution to the tantalizing question raised in the introduction on geometric classification of the behavior of $u$ at distinguished free boundary points.  When (H2) holds at a free boundary point, $0$, we will say that $0$ is a singular free boundary point of degree $|\alpha|$. In particular, regular free boundary points from classical study of cavitation problems, i.e. $\omega(0)$ finite, represent (singular) free boundary points of degree zero.

The weight $\omega_0$ is the homogenization limit, as it verifies $\omega_0(tx) = t^\alpha \omega_0(x)$, for all $t> 0$. Of course, should the original weight $\omega$ be homogeneous, then (H2) is immediately verified. 
%Also, we recall by reverse H\"older inequality $A_2$ weights belong to $L^{(1+\delta)}$.

The examples discussed at the end of Section \ref{sct math setup}  are, essentially, homogeneous. In addition to classical homogenization procedures, condition (H2) further contemplates perturbations of those by terms $g(x) = \text{o}(|x|^{-|\alpha|})$ as $|x| \to 0$, and products by bounded, positive functions $\theta(x) \in \text{VMO}$. So typically, we have in mind weights of the form
$$
	\omega(x) = \theta(x) \omega_0(x) + g(x),
$$
where $\omega_0$ is $\alpha$-homogeneous (as in the examples from Section \ref{sct math setup}),  $ \theta \in \text{VMO}$ verifying $0< \lambda_0< \theta(x) <\lambda_0^{-1}$, and  $|x|^{|\alpha|} g(x) \to 0$ as $|x| \to 0$.

\begin{lemma} \label{Lemma extra Hyp} Assume condition (H2) is in order. There exist constants $0< \tau_\star \le L < \infty$ such that
\begin{equation}\label{H3} 
	 \tau_\star r^{\alpha} \le  \intav{B_r(0)} \omega(x)dx \le L r^{\alpha},
\end{equation}
 or all $0< r \ll 1$. 
\end{lemma}
\begin{proof}
	For each $0< r \ll 1$ define the weight $\omega_r\colon B_1 \to \mathbb{R}$ as $\omega_r(y) = r^{|\alpha|} \omega(ry).$ It follows from condition (H2) that $\{\omega_r(y)\}_{r> 0}$ is a bounded set in $L^1(B_1)$. Hence, for some $L> 0$,
	$$
		L \ge  \intav{B_1(0)} \omega_r(x)dx =  r^{|\alpha|}  \intav{B_r(0)} \omega(x)dx,
	$$
	which shows the estimate from above. Now, let's suppose, seeking a contradiction, that no such a $\tau_\star> 0$ exists. It means one could find a sequence of radii $r_j =\text{o}(1)$ such that
	$$
		  \intav{B_{r_j}(0)} \omega(x)dx \le 2^{-j} r_j^{|\alpha|}.
	$$
	However, this would imply $\omega_{r_j}(y) = r_j^{|\alpha|} \omega(r_jy) \to 0$ in $L^1$, which contradicts condition (H2), as $0 \not \in A_2$.
\end{proof}
Before we continue, let us make a comment on the scaling of the free boundary problem \eqref{degFB}, which further substantiates (H2).
 \begin{remark}\label{scaling}
%----------------------------------------------------------------------------------------- SCALING
Let  $u$ be a local minimizer of 
$$
	\J(\omega, u,\Omega) := \int_{\Om} \omega(x) |\nabla u (x)|^2+\chi_{\{u>0\}} dx.
$$
Given  $0< \lambda < 1$, define 
$$
	\beta=1-\frac{\alpha}{2}, \quad  \quad u_\lambda(x) = \lambda^{-\beta} u(\lambda x), \quad  \text{and} \quad \omega_\lambda(x) = \lambda^{|\alpha|} \omega(\lambda x),
$$
then change of variables yields 
\begin{eqnarray*}
	\J(\omega, u,\Om)&=&\int_{\Om/\lambda}\left [ \omega(\lambda y) |\nabla u(\lambda y)|^2+\chi_{\{u(\lambda y)>0\}}\right]\lambda^d dy\\
	&=&\int_{\Om/\lambda}\left[\lambda^\alpha \omega_\lambda (y) \lambda^{2(-1+\beta)}|\nabla u_{\lambda}(y)|^2+\chi_{\{u_{\lambda}(y)>0\}}\right]\lambda^d dy \\
	&=& \lambda^d \J (\omega_\lambda, u_{\lambda},\Om/\lambda).
\end{eqnarray*}
That is, $u_\lambda$ is a local minimizer of functional $\J_\lambda$, ruled by an approximation of the homogenizing medium.   
%----------------------------------------------------------------------------------------- SCALING
\end{remark}

%We start off with a weak compactness result. 
%
%\begin{proposition}[Weak compactness in $W^{1,d+\mu}$]  \label{weak comp} Let $\lambda_k$ be any sequence converging to zero and $u_k$ local minima of $\J_k = \J(\omega_k, v, \Omega)$, with $\omega_k(x) := \lambda_k^{|\alpha|} \omega(\lambda_k x)$. There exists $\mu>0$ such that  $\{\nabla u_k\}_{k\ge 1}$ is locally weakly pre-compact in $L^{d+\mu}(\Omega)$. 
% \end{proposition}
%\begin{proof} We revisit the proof of Theorem \ref{reg0}, for $\omega = \omega_{\lambda_k}$, $u=u_k$ as to reach the corresponding of \eqref{thm reg0 eq6}, namely 
%	\begin{equation}\label{lemma compact L1 eq1} 
%	\int_{B_r(y)} |\nabla u_k| dx \le  C r^{d-1+\tau},
%\end{equation}
%for positive constants $C$ and $\tau>0$ independent of $k$.  Fixed $\Omega' \Subset \Omega$, by standard applications of H\"older inequality, we can bound $\{|\nabla u_k|\}_{k\in \mathbb{N}}$ in $L^{\frac{d}{1-\tau}}(\Omega')$, uniformly in $k$. We choose $\mu > 0$ such that $d(1-\tau)^{-1} = d+\mu$, and the result follows since $L^{d+\mu}$ is a reflexive space. 
%\end{proof}

In the sequel we indeed show that local minima of the functional $\J_\lambda$ converge to a minimizer of the singular homogenized problem ruled by $\omega_0$.

%%%%%%%%%%%%%%%%%%%%
\begin{theorem} \label{conv minima}
Assume 0 is a free boundary point and that $\omega$ satisfies (H1) and (H2). Let $\lambda_k$ be any sequence converging to zero and $u_k$ local minima of $\J_k = \J(\omega_k, v, \Omega)$, with $\omega_k(x) := \lambda_k^{|\alpha|} \omega(\lambda_k x)$. Then, up to a subsequence, $u_k$ converges locally uniformly to a local minimum of $\J(\omega_0, v, \Omega)$. 
\end{theorem}
\begin{proof} From Theorem \ref{reg0}, minimization properties and (H1), up to a subsequence, $u_k \to u_0$ locally uniformly in $\Omega$ and locally weakly in $W^{1,2}$.  Also, from (H2), $\omega^{1/2}_k \to \omega^{1/2}_0$ strongly in $L^{2}$, and thus  
$$
	\omega_k^{1/2} \nabla u_k \rightharpoonup \omega_0^{1/2} \nabla u_0,
$$
locally in the weak topology of $L^2$. Lower weak semicontinuity of norm along with Fatou's Lemma imply, for any fixed subdomain $\Omega' \Subset \Omega$, 
\begin{equation}\label{conv min0}
	\J(\omega_0, u_0, \Omega') \le \liminf \J(\omega_k, u_k, \Omega').
\end{equation}
Now, fix a ball $B:= B_r(x_0) \Subset \Omega$ and let $\varphi$ be a smooth function in $B$ satisfying $\varphi = u_0$ on $\partial B$. For $\epsilon > 0$ small, consider
$$
	\psi(y) := \frac{|y-x_0| - r}{\epsilon r}
$$
and define $\varphi^\epsilon_k \colon \Omega \to \mathbb{R}$ to be the linear interpolation between $u_k$ and $\varphi$ within $B_{(1+\epsilon)r}(x_0) $, that is:
\begin{equation}\label{varphi1}
\varphi^\epsilon_k(y) = \left\{	\begin{array}{ll}
		  \varphi(y), &\textrm{ if } |y-x_0|\le r \\
		 u_k(y), &\textrm{ if }  |y-x_0| \ge (1+\epsilon)r~ \\ %\label{varphi2}
		 (1-\psi(y)) u_0(y) +\psi(y) u_k(y), &\textrm{ if }  r< |y-x_0| < (1+\epsilon)r.~ % \label{varphi3}
	\end{array}\right.
\end{equation}
Since $u_k$ is a local minimizer of $\J_k$ over $\widetilde{B} := B_{(1+\epsilon)r}(x_0)$, we have 
$$\J(\omega_k, u_k, \widetilde{B}) \le \J(\omega_k, \varphi^\epsilon_k, \widetilde{B}),$$ 
that is
\begin{equation}\label{conv. min1}
	\J(\omega_k, u_k, \widetilde{B}) \le \J(\omega_k, \varphi, {B}) + \J(\omega_k, \varphi_k^\epsilon,  \widetilde{B}\setminus {B}).
\end{equation}
On the transition region, $r< |y-x_0| < (1+\epsilon)r$,  taking into account \eqref{varphi1}, we compute, for a.e. $y$,
$$
	\nabla \varphi^\epsilon_k(y) =  \nabla \psi(y) \cdot \left ( u_k(y) - u_0(y) \right )  + (1- \psi(y)) \nabla u_0(y) + \psi(y) \nabla u_k(y)
$$
which yields 
\begin{equation}\label{conv. min2}
	|\nabla \varphi^\epsilon_k(y)| \le |\nabla u_0(y)| + |\nabla u_k(y)| + \frac{C}{\epsilon} |u_k(y) - u_0(y)|.
\end{equation}
Consequently, taking into account that $|\widetilde{B} \setminus B| \sim \epsilon$, that $\omega_k(y) |\nabla u_k|^2$ is bounded in $L^{1+\delta}(\widetilde{B})$, for some $\delta>0$, and that $\omega_k(y) $ is bounded in $L^1$,  we obtain, from H\"older inequality
\begin{equation}\label{conv. min2.5}
	\begin{array}{lll}
	 	\J(\omega_k, \varphi_k^\epsilon,  \widetilde{B}\setminus {B}) &=& \displaystyle \int_{\tilde{B}\setminus B}  \omega_k(x) |\nabla \varphi_k^\epsilon|^2 + \chi_{\{\varphi_k^\epsilon > 0\}} dx \\
		& \le & C_1 |\widetilde{B} \setminus B|^{\frac{\delta}{1+\delta}} + C_2 \left ( \dfrac{C}{\epsilon} \right )^2 \|u_k - u_0\|_{L^\infty(\widetilde{B} \setminus B)} + \epsilon \\
		& \le & C\epsilon^\beta + \epsilon^{-2} \text{o}(1),
	\end{array} 
\end{equation}
where $0< \beta = \delta/(1+\delta) \ll 1$, is a small but positive number and  $\text{o}(1) = C  \|u_k - u_0\|_{L^\infty(\widetilde{B} \setminus B)}$ is an error that goes to zero as $k\to \infty$. Hence, combining \eqref{varphi1},  \eqref{conv. min1} and \eqref{conv. min2.5},
\begin{equation}\label{conv. min3}
\J(\omega_k, u_k, \widetilde{B}) \le  \J(\omega_k, \varphi, {B})   + C\epsilon^\beta + \epsilon^{-2} \text{o}(1),
\end{equation}
as $k \to \infty$. Finally, letting $k\to \infty$, taking into account \eqref{conv min0}, we obtain
$$
	\J(\omega_0, u_0, \tilde{B}) \le  \J(\omega_0, \varphi, {B}) + C\epsilon^\beta.
$$
Since $\varphi$ and $\epsilon> 0$ were taken arbitrary, we conclude the proof of the Theorem.
\end{proof}
%%%%%%%%%%%%%%%%%%%%

\section{ $C^{1+\gamma}$ regularity at the free boundary} \label{sct sharp reg}

In Section \ref{sct comp} we showed minimizers are locally H\"older continuous, which, in particular, yields a rough oscillation control of $u$ near the free boundary. The heart of the matter, though, is to describe the precise geometric behavior of a local minimizer at free boundary points. Roughly speaking, solutions to singular free boundary problems should adjust their vanishing rate based upon the singularity of the medium.

While there is no hope to obtain an estimate superior than H\"older continuity of local minima, at any other point, in this section we show that 
$u$ behaves as a $C^{1+\gamma}$ function around a singular free boundary point. As usual, this implies higher order differentiability of $u$ at free boundary points. In particular, if $1+\gamma = N$ is an integer, then $u \in C^{N-1, 1}$, in the sense it is $N-1$ differentiable, and the $N$th-Newtonian quotient remains bounded.  When $1+\gamma = N +\theta$, for $0< \theta < 1$, then solutions are $C^{N, \theta}$ regular at free boundary points -- see \cite{T4} for similar higher differentiability phenomenon.

This is the content of the following key result:
%-----------------------------------------------------------------------------------

\begin{theorem} \label{optimalRegAC}
Assume $(H1)-(H2)$  and let $u$ be a local minimizer of \eqref{degFB} so that 0 is a free boundary point. Then 
$$
	\sup\limits_{B_r} u \le C r^{1+ \frac{|\alpha|}{2}},
$$
for a universal  constant $C>0$, independent of $u$.
\end{theorem}

The proof of Theorem \ref{optimalRegAC} will be based on a geometric flatness improvement technique, in the spirit of \cite{T2, T3}. For that, we need:

\begin{lemma} \label{lemma F} Assume $(H1)-(H2)$  and let $u\in H^1(B_1,\omega)$ verify $0\le u \le 1$ in $B_1$, with $u(0) = 0$. Given $\delta > 0$, there exists $\epsilon>0$ depending only on $\delta$ and universal constants, such that if $u$ is a local minimizer of
\begin{equation}\label{Jepsilon}
	\J_\epsilon (u):= \int_{B_1} \omega_\lambda (x) |\nabla u (x)|^2+\epsilon \chi_{\{u>0\}} dx,
\end{equation}
for any $0< \lambda < 1$, then, in $B_{1/2}$, $u$ is at most $\delta$, that is,
\begin{equation}\label{conclusion lemma F}
	\sup\limits_{B_{1/2}} u \le \delta.
\end{equation}
\end{lemma}

\begin{proof}
	Suppose, for the sake of contradiction, that the thesis of the Lemma does not hold. This means for some $\delta_0>0$, one can find a sequence  of functions $u_j \in H^1(B_1,\omega_j)$ satisfying:
	\begin{enumerate}
		\item $0\le u_j \le 1$ in $B_1$;
		\item $u_j(0) = 0$;
		\item $u_j$ is a local minimizer of $\J_{j} (u):= \int_{B_1} \omega_{\lambda_j}(x) |\nabla u (x)|^2+\frac{1}{j} \chi_{\{u>0\}} dx$;
	\end{enumerate}
	however,
	\begin{equation}\label{eq1-lemma F}
		\sup\limits_{B_{1/2}} u_j \ge \delta_0 \quad \forall j \in \mathbb{N}.
	\end{equation}
	
From compactness estimates proven in Section \ref{sct comp},  up to a subsequence,  $u_j$ converges locally uniformly  to a nonnegative function $u_\infty$,  with $u_\infty (0) = 0$. Also, from similar analysis as in Theorem \ref{conv minima}, $u_\infty$ is a local minimizer of 
	$$
		\J_0(v) := \int_{B_1} \omega_0 (x) |\nabla v (x)|^2 dx.
	$$
	Applying maximum principle (available for minimizers of the functional $\J_0$), we conclude $u_\infty \equiv 0$. That is, we have proven $u_j$ converges locally uniformly to $0$. Therefore, for $j\gg 1$, we reach a contradiction with \eqref{eq1-lemma F}. The proof of Lemma \ref{lemma F} is complete.
\end{proof}
\smallskip
\noindent {\it Proof of Theorem \ref{optimalRegAC}}. We will make few (universal) decisions. Initially we set 
	$$
		\delta_\star := 2^{\frac{\alpha}{2} - 1}.
	$$ 
	Lemma \ref{lemma F} assures the existence of a positive (universal) constant $\epsilon_0>0$ such that any normalized minimizer of
	$\J_{\epsilon_0}$, as defined in \eqref{Jepsilon} verifies \eqref{conclusion lemma F}, for $\delta_\star$. Define
	$$
		\varrho_0 := \sqrt[2-\alpha]{\epsilon_0}, \quad \text{and} \quad \tilde{u}(x) := u(\varrho_0 x).
	$$
	By remark \ref{scaling}, $\tilde{u}$ is a local minimizer of $\J_{\epsilon_0}$, and hence, from Lemma \ref{lemma F}, there holds:
	\begin{equation}\label{reg-S1}
		\sup\limits_{B_{1/2}} \tilde{u}(x) \le 2^{\frac{\alpha}{2} - 1}.
	\end{equation}
	Next, by induction, we iterate the previous argument to show 
	\begin{equation}\label{reg-Sk}
		\sup\limits_{B_{2^{-k}}} \tilde{u}(x) \le 2^{k\left (\frac{\alpha}{2} - 1\right )}.
	\end{equation}
	Estimate \eqref{reg-S1} gives the first step of induction, $k=1$. Now, suppose we have verified \eqref{reg-Sk} for $k=1,2,\cdots p$. Define
	$$
		\tilde{v}(x) := 2^{p\left (1 - \frac{\alpha}{2} \right )} \cdot \tilde{u} \left ( 2^{-p} x \right ).
	$$
	It follows from induction hypothesis that $0\le v \le 1$. Also, from scaling, we check that $\tilde{v}$ is also a minimizer of $\J_{\epsilon_0}$. Applying Lemma \ref{lemma F} to $\tilde{v}$ we conclude
	$$
		\sup\limits_{B_{1/2}} \tilde{v} \le 2^{\frac{\alpha}{2} - 1},
	$$
	which, in terms of $\tilde{u}$, gives  precisely the $(p+1)$ step of induction. 
	
	Now, given a (universally small) radius $r>0$, choose $k\in \mathbb{N}$, such that 
	$$
		2^{-(k+1)} < \varrho_0^{-1} r \le 2^{-k}.	
	$$
	We can then estimate
	$$
		\begin{array}{lll}
			\sup\limits_{B_r} u &=& \sup\limits_{B_{ \varrho_0^{-1} r}} \tilde{u} \\
			&\le & \sup\limits_{B_{2^{-k}}} \tilde{u} \\
			&\le & 2^{k\left (\frac{\alpha}{2} - 1\right )} \\
			&\le &\left ( \dfrac{2}{\varrho_0}\right )^{1-\frac{\alpha}{2}} \cdot r^{1+\frac{|\alpha|}{2}} \\
			&=& C r^{1+\frac{|\alpha|}{2}},
		\end{array}
	$$
	for $C>1$ universal, as required. \hfill $\square$

\section{Nondegeneracy and weak geometry} \label{sct nondeg}

In the previous section we show local minima are $C^{1+\gamma}$ smooth along the singular free boundary, for $\gamma = \frac{|\alpha|}{2}$. In this section we prove  a competing inequality which assures that such a geometric decay is sharp.  Surprisingly enough, for such an estimate one only needs the corresponding upper bound for the degree of singularity of the free boundary and not the full condition (H2). Thus, for didactical purposes we state it as a separate condition, which, in accordance to Lemma \ref{Lemma extra Hyp}, is implied by condition (H2).

\begin{itemize}
\item[(H3)] For some $-d< \alpha \le 0$, there exists ${L}>0$ such that
	$$
		\intav{B_r(0)} \omega(x)dx \le L r^\alpha
	$$
	for all $0< r \ll 1$.
 \end{itemize}

Again, we recall $0$ is the localized free boundary point at where we are analyzing the geometric behavior of $u$. Thus, condition (H3) simply conveys the idea that the origin is a singular free boundary point of degree {\it at most} $|\alpha|$. Note that (H3) yields
\begin{equation}\label{commH3}
	\sup\limits_{r>0}  \intav{B_1} \omega_r(x)dx \le {L},
\end{equation}
where, as before, $\omega_r(x) = r^{|\alpha|} \omega (rx)$.
%-----------------------------------------------------------------------------------
\begin{theorem}\label{nonDegAC} Let $u$  be a minimizer of \eqref{degFB}, $0$ a free boundary point and  assume (H3). Then
$$
	\sup\limits_{B_r} u(x) \geq 2 \cdot \sqrt{ \dfrac{1}{L} \dfrac{d^d}{(d+2)^{d+2}} } \cdot  r^{1+ \frac{|\alpha|}{2}}, \quad \forall 0<r<1,
$$
where $d$ is dimension. 
\end{theorem}
%-----------------------------------------------------------------------------------

\begin{proof} The idea of the proof is to cut a family of concentric holes on the graph of $u$, compare the resulting functions with $u$ in terms of the minimization problem $\J$, and finally optimize the cutting-hole parameter; here are the details. Let $0 < r <1$ be a fixed radius and define $v_r \colon B_1 \to \mathbb{R}$ as
$$
	v_r(y) := r^{\frac{\alpha}{2} -1 } u( r y).
$$ 
The goal is to show that 
$$
	\mathscr{S}_r := \sup\limits_{B_1} v_r
$$
is uniformly bounded from below, independently of $r$. From Remark \ref{scaling}, $v_r$ is a local minimizer of $\J_r$ over $B_1$. Next, let us choose $0<\sigma < 1$, $0<\epsilon \ll 1$ and craft a smooth, radially symmetric function $\varphi \colon B_1 \to \mathbb{R}$  satisfying:
\begin{equation}
0 \le \varphi \le 1, \quad \varphi \equiv 0 \text{ in } B_{\sigma}, \quad \varphi = 1 \text{ on } \partial B_1, \quad |\nabla \varphi| \le  (1+\epsilon)(1-\sigma)^{-1}.
\end{equation}
In the sequel, let us consider the test function $\xi \colon B_1 \to \mathbb{R}$ given by
$$
	\xi(x) := \min \left \{ v_r(x), (1+\epsilon) \mathscr{S}_r \cdot \varphi(x) \right \}.
$$
By construction, $\xi$ competes with $v_r$ in the minimization problem $\J_r$, and thus 
\begin{equation}\label{nondeg eq001}
	\int_{B_1}\left (  \omega_r(x)  |\nabla \xi|^2 + \chi_{\left \{ \xi >0\right \}}  \right )  dx  \ge \int_{B_1} \left (   \omega_r(x)  |\nabla v_r|^2 +  \chi_{\left \{ v >0\right \}}   \right ) dx.
\end{equation}
We can rewrite \eqref{nondeg eq001} as an inequality of the form $\mathcal{A} \ge \mathcal{B}$, for
\begin{equation}\label{nondeg eq 0015}
	\begin{array}{ll}
		\mathcal{A} := & \displaystyle \int_{B_1} \omega_r(x) \left ( |\nabla \xi|^2-|\nabla v|^2 \right ) dx,  \\
		\mathcal{B} :=  &  \displaystyle  \int_{B_1} \left ( \chi_{\left \{ v_r >0\right \}}-\chi_{\left \{ \xi >0\right \}}  \right ) dx.
	\end{array}
\end{equation}
As to estimate $\mathcal{B}$ from below we note that $\left \{ v >0\right \} \supset \left \{ \xi >0\right \}$, thus  
\begin{equation}\label{nondeg eq 0018}
	 \mathcal{B} = \displaystyle  \int_{B_1} \chi_{\{\xi = 0\}} dx \ge |B_\sigma|.
\end{equation}
Next we estimate $\mathcal{A}$ from above. For that, let us define 
$$
	\Pi := \left \{x \in B_1 \suchthat (1+\epsilon) \mathscr{S}_r \cdot \varphi(x) < v_r(x) \right \}
$$
and compute
\begin{equation}\label{nondeg eq002}
	\begin{array}{lll}
		\displaystyle \int_{B_1} \omega_r(y) \left ( |\nabla \xi|^2-|\nabla v_r|^2 \right ) dx  & = & \displaystyle \int\limits_{ \Pi } \omega_r(x) \left ( |\nabla \xi|^2-|\nabla v_r|^2 \right ) dx \\
		& \le &  (1+\epsilon)^2  \mathscr{S}_r^2 \displaystyle \int\limits_{B_1} \omega_r(x) \  |\nabla \varphi|^2  dx \\
		& \le &   (1+\epsilon)^4 (1-\sigma)^{-2} \cdot \left (\displaystyle  \int\limits_{B_1} \omega_r(x) dx  \right ) \cdot  \mathscr{S}_r^2 \\
		& \le &  (1+\epsilon)^4 (1-\sigma)^{-2} \cdot L |B_1| \cdot  \mathscr{S}_r^2,
	\end{array}
\end{equation}
where in the last inequality we have used condition (H3), along with comment \eqref{commH3}. From the relation $\mathcal{A} \ge \mathcal{B}$, we obtain
$$
	 \begin{array}{lll}
	 	 \mathscr{S}_r^2 &\ge& \dfrac{|B_\sigma|}{L|B_1|} \cdot \dfrac{(1-\sigma^2)}{(1+\epsilon)^4} \\
		 			&=& \dfrac{1}{L(1+\epsilon)^4} \cdot \dfrac{\sigma^d}{(1-\sigma^2)}.
	\end{array}
$$
Letting $\epsilon \to 0$ and selecting $\sigma = \dfrac{d}{d+2}$ yields the optimal lower bound. The proof of Theorem \ref{nonDegAC} is complete.
\end{proof}

An important consequence of Theorem \ref{optimalRegAC} and Theorem \ref{nonDegAC} combined is that, around free boundary points of same homogeneity $\alpha$, a local minimum detaches from its coincidence set, $\{u=0\}$, precisely as $\dist^{{1+ \frac{|\alpha|}{2}}}$.

\begin{corollary} 
Let $u$ be a minimizer of \eqref{degFB} and assume all free boundary points in $B_{1/2}$ satisfy conditions (H1) and (H2). Then there exists a universal constant $C>1$ such that
$$
	  C^{-1} \dist(x,\partial \left \{ u>0 \right \} )^{1+ \frac{|\alpha|}{2}}\le u(x) \le C \cdot  \dist(x,\partial \left \{ u>0 \right \} )^{1+ \frac{|\alpha|}{2}}
$$
 for any $x \in B_{1/4} \cap \{u > 0 \}$.
\end{corollary}
By standard arguments we then conclude that near  free boundary points of same homogeneity, the set of positivity $\{u> 0 \}$ has uniform positive density:
\begin{corollary} 
Let $u$ be a minimizer of \eqref{degFB} and assume all free boundary points in $B_{1/2}$ satisfy (H1) and (H2). Then, for any $0< r< \frac{1}{4}$,
$$
	  \dfrac{\left | B_r \cap \{u > 0 \} \right |}{ \left | B_r \right |} \ge C,
$$
where $0<C< 1$ is a universal constant, independent of $r$.
\end{corollary}


\begin{thebibliography}{99}   %


\bibitem{AC} Caffarelli, L. and Alt, H. {\it Existence and regularity for a minimum problem with free boundary.}  (1981) J. Reine Angew. Math. 325, 105--144.

\bibitem{CRS} Caffarelli, L., Roquejoffre, J-M and Sire, Y.  {\it Variational problems for free boundaries for the fractional Laplacian.} J. Eur. Math. Soc. (JEMS) 12 (2010), no. 5, 1151-1179. 

\bibitem{cafS}  Caffarelli, L. and Silvestre, L. {\it An extension problem related to the fractional Laplacian}.  Comm. Partial Differential Equations 32 (2007), no. 7-9, 1245--1260.

\bibitem{deSilva} De Silva, D. 
{\it Free boundary regularity for a problem with right hand side.} Interfaces Free Bound. 13 (2011), no. 2, 223--238. 


\bibitem{dosPT}  Dos Prazeres, D. and Teixeira, E.   {\it Cavity problems in discontinuous media.} Calc. Var. Partial Differential Equations 55 (2016), no. 1, Art. 10, 15 pp.

\bibitem{FKS}  Fabes, E.,  Kenig, C. and Serapioni, R.  {\it The local regularity of solutions of degenerate elliptic equations.} Comm. Partial Differential Equations 7 (1982), no. 1, 77--116.

\bibitem{FJK} Fabes, E., Jerison, D. and Kenig, C.
{\it The Wiener test for degenerate elliptic equations.}  Ann. Inst. Fourier (Grenoble) 32 (1982), no. 3, vi, 151--182. 

\bibitem{FJK2} Fabes, E., Jerison, D. and Kenig, C.
{\it Boundary behavior of solutions to degenerate elliptic equations.} Conference on harmonic analysis in honor of Antoni Zygmund, Vol. I, II (Chicago, Ill., 1981), 577--589, 
 
 
 \bibitem{FS} Ferrari, F.; Salsa, S.  
{\it Regularity of the free boundary in two-phase problems for linear elliptic operators.} 
Adv. Math. 214 (2007), no. 1, 288--322. 

\bibitem{G} G\'orka, Przemyslaw
{\it Campanato theorem on metric measure spaces.} 
Ann. Acad. Sci. Fenn. Math. 34 (2009), no. 2, 523--528.

%\bibitem{G-Book}  Giusti, Enrico Direct methods in the calculus of variations. World Scientific Publishing Co., Inc., River Edge, NJ, 2003. viii+403 pp. ISBN: 981-238-043-4.

%\bibitem{HL} Han, Q.; Lin, F. 
%{\it Elliptic partial differential equations.}
%Second edition. Courant Lecture Notes in Mathematics, 1. Courant Institute of Mathematical Sciences, New York; American Mathematical Society, Providence, RI, 2011. x+147 pp. ISBN: 978-0-8218-5313-9 

\bibitem{HKM} Heinonen, Juha; Kilpel\"ainen, Tero; Martio, Olli {\it Nonlinear potential theory of degenerate elliptic equations.}
Unabridged republication of the 1993 original. Dover Publications, Inc., Mineola, NY, 2006. xii+404 pp. ISBN: 0-486-45050-3 

 \bibitem{MT}  Moreira, D.; Teixeira, E. 
{\it On the behavior of weak convergence under nonlinearities and applications.} 
Proc. Amer. Math. Soc. 133 (2005), no. 6, 1647--1656. 
 
 
 
%\bibitem{Morrey} Morrey, C.  {\it Multiple Integrals in the Calculus of Variations.} Grundlehren Math. Wiss.
%130, Springer, New York (1966).

\bibitem{muck} Muckenhoupt, B {\it Weighted norm inequalities for the Hardy maximal function}. Trans. Amer. Math. Soc. (1972) 165 207--226.

%\bibitem{Stein} Stein, Elias M. 
%{\it Harmonic analysis: real-variable methods, orthogonality, and oscillatory integrals. }
%Princeton Mathematical Series, 43. Monographs in Harmonic Analysis, III. Princeton University Press, Princeton, NJ, 1993. xiv+695 pp. ISBN: 0-691-03216-5 

\bibitem{T1}  Teixeira, E. {\it A variational treatment for general elliptic equations of the flame propagation type: regularity of the free boundary.} 
Ann. Inst. H. Poincaré Anal. Non Lin\'eaire 25 (2008), no. 4, 633--658. 


\bibitem{T2} Teixeira, E. {\it Regularity for quasilinear equations on degenerate singular sets. } Math. Ann. 358 (2014), no. 1-2, 241--256.

\bibitem{T2.5} Teixeira, E. {\it Sharp regularity for general Poisson equations with borderline sources.} J. Math. Pures Appl. (9) 99 (2013), no. 2, 150--164.

\bibitem{T3} Teixeira, E. {\it Universal moduli of continuity for solutions to fully nonlinear elliptic equations.} Arch. Ration. Mech. Anal. 211 (2014), no. 3, 911--927. 

\bibitem{T4} Teixeira, E.
{\it Regularity for the fully nonlinear dead-core problem.}
Math. Ann. 364 (2016), no. 3-4, 1121--1134. 

\end{thebibliography}
\end{document}